\documentclass{amsart}
\usepackage{amsmath, amsfonts, amssymb,amscd,graphics,graphicx,color,eucal,setspace}
\usepackage{latexsym}
\usepackage{color}
\usepackage[top=3.5cm, bottom=2.5cm, left=2.5cm, right=2.5cm, includefoot]{geometry}

\usepackage[all]{xy}
\usepackage{url}
\usepackage{amscd}
\usepackage{here}
\numberwithin{equation}{section}
\usepackage{bm}
\usepackage{tikz}

\theoremstyle{plain}
\newtheorem{theorem}{Theorem}[section]
\newtheorem{lemma}[theorem]{Lemma}
\newtheorem{corollary}[theorem]{Corollary}
\newtheorem{proposition}[theorem]{Proposition}

\theoremstyle{definition}
\newtheorem{remark}{Remark}[section]
\newtheorem{example}{Example}[section]
\newtheorem{definition}[theorem]{Definition}

\newcommand{\Flag}{\mathrm{Fl}}
\def\C{\mathbb C}
\def\R{\mathbb R}
\def\Q{\mathbb Q}
\def\Z{\mathbb Z}

\DeclareMathOperator{\Hess}{Hess}

\DeclareMathOperator{\Poin}{Poin}
\DeclareMathOperator{\Hilb}{Hilb}

\def\Sn{\mathfrak{S}_n}

\def\Setdef#1|#2\Setdef{\left\{#1\;\mathstrut\vrule\;#2\right\}}%

\def\rsetdef#1|#2\rsetdef{\!\left(#1_\bullet\;\mathstrut\vrule\;#2\right)}%

\def\tausetdef#1|#2\tausetdef{\!\left(\sum\limits_{\substack{\bullet \in A \\ |A| < #1}}\tau_A\;\mathstrut\vrule\;#2\right)}%

\definecolor{gray7}{gray}{0.7}

\def\L{\textsf{\upshape L}}
\def\J{\bot}

\def\Fl{{\rm Fl}}
\def\g{\mathbf{g}}
\def\1{\mathbf{1}}

\makeatletter
\@namedef{subjclassname@2020}{%
  \textup{2020} Mathematics Subject Classification}
\makeatother

\begin{document}
\title[Regular semisimple Hessenberg varieties]
{Regular semisimple Hessenberg varieties with cohomology rings generated in degree two}

\author[M. Masuda]{Mikiya Masuda}
\address{Osaka Metropolitan University Advanced Mathematical Institute, Sumiyoshi-ku, Osaka 558-8585, Japan.}
\email{mikiyamsd@gmail.com}

\author[T. Sato]{Takashi Sato}
\address{Osaka Metropolitan University Advanced Mathematical Institute, Sumiyoshi-ku, Osaka 558-8585, Japan.}
\email{00tkshst00@gmail.com}

\date{\today}



\begin{abstract}
A regular semisimple Hessenberg variety $\Hess(S,h)$ is a smooth subvariety of the flag variety determined by a square matrix $S$ with distinct eigenvalues and a Hessenberg function $h$.
The cohomology ring $H^*(\Hess(S,h))$ is independent of the choice of $S$ and is not explicitly described except for a few cases.
In this paper, we characterize the Hessenberg function $h$ such that $H^*(\Hess(S,h))$ is generated in degree two as a ring.
It turns out that such $h$ is what is called a (double) lollipop. 
\end{abstract}

\maketitle

\setcounter{tocdepth}{1}

\section{Introduction} \label{sect:Intro}
The flag variety $\Fl(n)$ consists of nested sequences of linear subspaces in the complex vector space $\C^n$:
\[
	\Fl(n)=\{V_\bullet=(V_1\subset V_2\subset \cdots\subset V_n=\C^n)
	\mid \dim_\C V_i=i\quad (\forall i\in [n]=\{1,2,\dots,n\})\}.
\]
A Hessenberg function $h\colon [n]\to [n]$ is a monotonically non-decreasing function satisfying $h(j) \geq j$ for any $j \in [n]$.
We often express a Hessenberg function $h$ as a vector $(h(1),\dots,h(n))$ by listing the values of $h$.
Given an $n\times n$ matrix $A$ and a Hessenberg function $h$,
a Hessenberg variety $\Hess(A,h)$ is defined as 
\[
	\Hess(A,h)=\{ V_\bullet\in \Fl(n)\mid AV_i\subset V_{h(i)}\quad (\forall i\in [n])\}
\]
where the matrix $A$ is regarded as a linear operator on $\C^n$.
Note that $\Hess(A,h)=\Fl(n)$ if $h=(n,\dots,n)$. 

The family of Hessenberg varieties $\Hess(A,h)$ contains important varieties such as Springer fibers ($A$ is nilpotent and $h=(1,2,\dots,n)$),
Peterson varieties ($A$ is regular nilpotent and $h=(2,3,\dots,n,n)$),
and permutohedral varieties ($A$ is regular semisimple and $h=(2,3,\dots,n,n)$),
which are toric varieties with permutohedra as moment polytopes. 

Among $n\times n$ matrices, regular semisimple ones $S$ (i.e.\ matrices $S$ having distinct eigenvalues) are generic and $\Hess(S,h)$ is called a regular semisimple Hessenberg variety.
The regular semisimple Hessenberg variety $\Hess(S,h)$ has nice properties.
For instance, it is smooth and its cohomology $H^*(\Hess(S,h))$ becomes a module over the symmetric group $\mathfrak{S}_n$ on $[n]$ by Tymoczko's dot action \cite{tymo08}.
Remarkably, the solution of Shareshian--Wachs conjecture \cite{sh-wa16} by Brosnan and Chow \cite{br-ch} (and Guay-Paquet \cite{guay}) connected $H^*(\Hess(S,h))$ as an $\mathfrak{S}_n$-module and chromatic symmetric functions on certain graphs.
This opened a way to prove the famous Stanley--Stembridge conjecture in graph theory through the geometry or topology of Hessenberg varieties and motivated us to study $H^*(\Hess(S,h))$.
Note that $H^*(\Hess(S,h))$ (indeed the diffeomorphism type of $\Hess(S,h)$) is independent of the choice of $S$.
We write the regular semisimple Hessenberg variety $\Hess(S,h)$ as $X(h)$ for brevity since our concern in this paper is its cohomology ring.

The $\mathfrak{S}_n$-module structure on $H^*(X(h))$ is determined in some cases (e.g.\ \cite{ha-pr19}).
In particular, that on $H^2(X(h))$ was explicitly described
by Chow \cite{chow21} combinatorially (through the theorem by Brosnan-Chow mentioned above)
and by Cho-Hong-Lee \cite{CHL21} geometrically.
Motivated by their works, Ayzenberg and the authors \cite{AMS1} reproved their results
by giving explicit additive generators of $H^2(X(h))$ in terms of GKM theory.

The ring structure on $H^*(X(h))$ is not explicitly described except for a few cases.
Remember that $X(h)$ for $h=(n,\dots,n)$ is the flag variety $\Fl(n)$ and $H^*(\Fl(n))$ is generated in degree $2$ as a ring.
Moreover, $X(h)$ for $h=(2,3,\dots,n,n)$ is the permutohedral variety and $H^*(X(h))$ is also generated in degree $2$ as a ring.
On the other hand, for $h=(h(1),n,\dots,n)$ with $h(1)$ arbitrary,
a result of \cite{AHM} shows that $H^*(X(h))$ is generated in degree $2$ as a ring if and only if $h(1)=2$ or $n$,
where $X(h)=\Fl(n)$ for the latter case.
Therefore, it is natural to ask when $H^*(X(h))$ is generated in degree $2$ as a ring.
The answer is the following, which is our main result in this paper. 

\begin{theorem}\label{thm:main1}
Assume that $h(j)\ge j+1$ for $j\in [n-1]$.
Then $H^*(X(h))$ is generated in degree $2$ as a ring if and only if $h$ is of the following form \eqref{eq:h} for some $1\le a< b\le n$,
\begin{equation} \label{eq:h}
	h(j)=
	\begin{cases}
		a+1 & (1\le j\le a)\\
		j+1 & (a< j< b)\\
		n & (b\le j\le n).
	\end{cases}
\end{equation}
\end{theorem}


\begin{remark}
\begin{enumerate}
\item
$X(h)$ is connected if and only if $h(j)\ge j+1$ for any $j\in [n-1]$.
When $X(h)$ is not connected, each connected component of $X(h)$ is a product of smaller regular semisimple Hessenberg varieties. 
\item
$X(h)$ is the flag variety $\Flag(n)$ when $(a,b)=(n-1,n)$ and is the permutohedral variety when $(a,b)=(1,n)$.
\item
We will give an explicit presentation of the ring structure on $H^*(X(h))$ for $h$ of the form \eqref{eq:h} in a forthcoming paper \cite{ma-sa}. 
\end{enumerate}
\end{remark}

We can visualize a Hessenberg function $h$ by drawing a configuration of the shaded boxes on a square grid of size $n\times n$,
which consists of boxes in the $i$-th row and the $j$-th column satisfying $i\le h(j)$.
Since $h(j)\ge j$ for any $j\in [n]$, the essential part is the shaded boxes below the diagonal.
For example, Figure \ref{pic:stair-shape} below is the configurations of two Hessenberg functions $h$ of the form \eqref{eq:h} with $n=11$:
one is $h=(2,3,4,5,6,7,11,11,11,11)$ where $(a,b)=(1,7)$
and the other is $h=(4,4,4,5,6,7,11,11,11,11)$ where $(a,b)=(3,7)$.
We often identify a Hessenberg function $h$ with its configuration. 

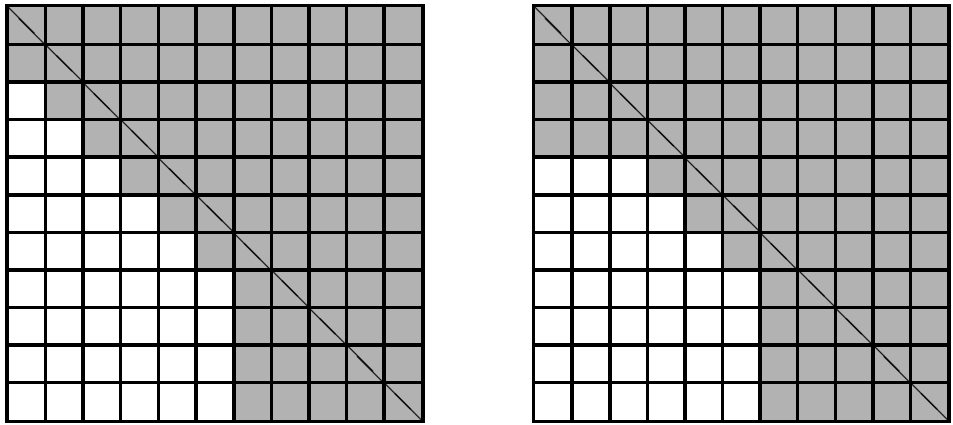
\begin{figure}[H]
\begin{center}
\setlength{\unitlength}{5mm}
\begin{picture}(25,11)(0,0)
\multiput(0,10.2)(0,-1){2}{\colorbox{gray7}{\phantom{\vrule width 3mm height 3mm}}}
\multiput(1,10.2)(0,-1){3}{\colorbox{gray7}{\phantom{\vrule width 3mm height 3mm}}}
\multiput(2,10.2)(0,-1){4}{\colorbox{gray7}{\phantom{\vrule width 3mm height 3mm}}}
\multiput(3,10.2)(0,-1){5}{\colorbox{gray7}{\phantom{\vrule width 3mm height 3mm}}}
\multiput(4,10.2)(0,-1){6}{\colorbox{gray7}{\phantom{\vrule width 3mm height 3mm}}}
\multiput(5,10.2)(0,-1){7}{\colorbox{gray7}{\phantom{\vrule width 3mm height 3mm}}}
\multiput(6,10.2)(0,-1){11}{\colorbox{gray7}{\phantom{\vrule width 3mm height 3mm}}}
\multiput(7,10.2)(0,-1){11}{\colorbox{gray7}{\phantom{\vrule width 3mm height 3mm}}}
\multiput(8,10.2)(0,-1){11}{\colorbox{gray7}{\phantom{\vrule width 3mm height 3mm}}}
\multiput(9,10.2)(0,-1){11}{\colorbox{gray7}{\phantom{\vrule width 3mm height 3mm}}}
\multiput(10,10.2)(0,-1){11}{\colorbox{gray7}{\phantom{\vrule width 3mm height 3mm}}}
\linethickness{0.3mm}
\multiput(1,0)(1,0){10}{\line(0,1){11}}
\multiput(0,1)(0,1){10}{\line(1,0){11}}
\multiput(0,11)(1,-1){11}{\line(1,-1){1}}
\put(0,0){\framebox(11,11)}
\multiput(14,10.2)(0,-1){4}{\colorbox{gray7}{\phantom{\vrule width 3mm height 3mm}}}
\multiput(15,10.2)(0,-1){4}{\colorbox{gray7}{\phantom{\vrule width 3mm height 3mm}}}
\multiput(16,10.2)(0,-1){4}{\colorbox{gray7}{\phantom{\vrule width 3mm height 3mm}}}
\multiput(17,10.2)(0,-1){5}{\colorbox{gray7}{\phantom{\vrule width 3mm height 3mm}}}
\multiput(18,10.2)(0,-1){6}{\colorbox{gray7}{\phantom{\vrule width 3mm height 3mm}}}
\multiput(19,10.2)(0,-1){7}{\colorbox{gray7}{\phantom{\vrule width 3mm height 3mm}}}
\multiput(20,10.2)(0,-1){11}{\colorbox{gray7}{\phantom{\vrule width 3mm height 3mm}}}
\multiput(21,10.2)(0,-1){11}{\colorbox{gray7}{\phantom{\vrule width 3mm height 3mm}}}
\multiput(22,10.2)(0,-1){11}{\colorbox{gray7}{\phantom{\vrule width 3mm height 3mm}}}
\multiput(23,10.2)(0,-1){11}{\colorbox{gray7}{\phantom{\vrule width 3mm height 3mm}}}
\multiput(24,10.2)(0,-1){11}{\colorbox{gray7}{\phantom{\vrule width 3mm height 3mm}}}
\linethickness{0.3mm}
\multiput(15,0)(1,0){10}{\line(0,1){11}}
\multiput(14,1)(0,1){10}{\line(1,0){11}}
\multiput(14,11)(1,-1){11}{\line(1,-1){1}}
\put(14,0){\framebox(11,11)}
\end{picture}
\end{center}
\caption{The configurations for $h = (2,3,4,5,6,7,11,11,11,11)$ and $h = (4,4,4,5,6,7,11,11,11,11)$}
\label{pic:stair-shape}
\end{figure}

The chromatic symmetric functions and LLT polynomials associated with $h$ of the form \eqref{eq:h} are studied from the viewpoint of combinatorics in \cite{da-wi, hu-na-yo},
and when $a=1$ or $b=n$, the corresponding Hessenberg functions
\[
	h=(2,3,\dots,b,n,\dots,n)\quad \text{or}\quad (a+1,\dots,a+1,a+2,\dots,n-1,n,n)
\]
are called lollipops in those papers, so the Hessenberg function of the form \eqref{eq:h} may be called a double lollipop. 

The paper is organized as follows.
In Section 2, we review GKM theory to compute the equivariant cohomology of $X(h)$.
We prove the ``only if'' part in Theorem \ref{thm:main1} in Section 3.
Indeed, we consider a Morse-Bott function $f_h$ on $X(h)$, where the inverse image of the minimum or maximum value of $f_h$ is a regular semisimple Hessenberg variety $X(h')$ with $h'$ of size one less than that of $h$.
Then a property of the Morse-Bott function $f_h$ shows the surjectivity of the restriction map $H^*(X(h);\Q)\to H^*(X(h');\Q)$,
and this enables us to use an inductive argument to prove the ``only if'' part.
In Section 4, we prove the ``if'' part in Theorem \ref{thm:main1} by
showing that $X(h)$ is the total space of a fiber bundle over a compact smooth toric variety with a product of flag varieties as the fiber when $h$ is of the form \eqref{eq:h}.
Since the cohomology rings of the base space and the fiber are both generated in degree two, so is the cohomology ring of $X(h)$.

\section{Regular semisimple Hessenberg varieties} \label{sect:Object}

We first recall some properties of a regular semisimple Hessenberg variety $X(h)$.
\begin{theorem}[\cite{ma-pr-sh}]
\label{thm:MPS}\hfill
\begin{enumerate}
	\item
	$X(h)$ is smooth.
	\item
	$\dim_\C X(h)=\sum_{j=1}^n(h(j)-j)$.
	\item
	$X(h)$ is connected if and only if $h(j)\ge j+1$ for $\forall j\in [n-1]$.
	\item
	$H^{odd}(X(h))=0$ and the $2k$-th Betti number of $X(h)$ is given by
	\[
		\#\{w\in \mathfrak{S}_n\mid \ell_h(w)=k\}
	\]
	where
	\begin{equation}\label{eqDhw}
		\ell_h(w)=\#\{ 1\le j<i\le n\mid w(j)>w(i),\ i\le h(j)\}.
	\end{equation}
\end{enumerate}
\end{theorem}

For calculation of the cohomology ring of $X(h)$, we use equivariant cohomology which we shall explain.
We assume that the matrix $S$ in $X(h)=\Hess(S,h)$ is a diagonal matrix.
Let $T$ be an algebraic torus consisting of diagonal matrices in the general linear group ${\rm GL}_n(\C)$.
The linear action of $T$ on $\C^n$ induces an action on the flag variety $\Fl(n)$ and preserves $X(h)$ since $S$ commutes with $T$.
The fixed point sets of the $T$-actions on $X(h)$ and $\Flag(n)$ consist of all permutation flags, that is, 
\begin{equation} \label{eq:fixed_point_set}
	X(h)^{T} = \Flag(n)^{T} \cong \mathfrak{S}_n.
\end{equation} 

Since $T$ can naturally be identified with $(\C^*)^n$, the classifying space $BT$ of $T$ is $B(\C^*)^n=(\C P^\infty)^n$.
Let $p_i\colon T\to \C^*$ be the projection on the $i$-th diagonal component of $T$ and $t_i=p_i^*(t)\in H^2(BT)$
where $p_i^*\colon H^*(B\C^*)\to H^*(BT)$ and $t\in H^2(B\C^*)$ is the first Chern class of the tautological line bundle over $B\C^*=\C P^\infty$.
Then 
\begin{equation} \label{eq:H*BT}
	H^*(BT)=\Z[t_1,\dots,t_n].
\end{equation}

The equivariant cohomology of the $T$-variety $X(h)$ is defined as 
\[
	H^*_T(X(h)):=H^*(ET\times_T X(h))
\]
where $ET$ is the total space of the universal principal $T$-bundle $ET\to BT$ and $ET\times_T X(h)$ is the orbit space of the product $ET\times X(h)$ by the diagonal $T$-action.
The projection $ET\times X(h)\to ET$ on the first factor induces a fibration 
\begin{equation*} \label{eq:fibration}
	X(h)\xrightarrow{\rho} ET\times_T X(h)\xrightarrow{\pi} BT. 
\end{equation*}
Since $H^{odd}(X(h))=0$ as in Theorem \ref{thm:MPS} and $H^{odd}(BT)=0$,
the Serre spectral sequence of the fibration above collapses.
It implies that $\rho^*\colon H^*_T(X(h))\to H^*(X(h))$ is surjective and induces a graded ring isomorphism 
\begin{equation} \label{eq:quotient_of_equivariant_cohomology}
	H^*(X(h))\cong H^*_T(X(h))/(\pi^*(t_1),\dots,\pi^*(t_n))
\end{equation}
by \eqref{eq:H*BT}.
Therefore, one can find the ring structure on $H^*(X(h))$ through $H^*_T(X(h))$. 

Since $H^{odd}(X(h))=0$, it follows from the localization theorem that the homomorphism 
\begin{equation} \label{eq:restriction_hom}
	H^*_T(X(h)) \to H^*_T(X(h)^T)= \bigoplus_{w\in \Sn}H^*_T(w)
	=\bigoplus_{w\in \Sn}\Z[t_1,\dots,t_n]
	={\rm Map}(\Sn,\Z[t_1,\dots,t_n])
\end{equation}
induced from the inclusion map $X(h)^T\to X(h)$ is injective,
where $X(h)^T$ is identified with $\mathfrak{S}_n$ by \eqref{eq:fixed_point_set} and ${\rm Map}(P,Q)$ denotes the set of all maps from $P$ to $Q$.
The $T$-variety $X(h)$ is what is called a GKM manifold and the image of the homomorphism in \eqref{eq:restriction_hom} is described in \cite{tymo08} as follows;
\begin{equation} \label{eq:GKM_condition_for_X(h)}
	\{ f \in {\rm Map}(\Sn,\Z[t_1,\dots,t_n]) \mid f(w) - f(w(i,j)) \in (t_{w(i)} - t_{w(j)}),
	\text{ for } \forall w \in \mathfrak S_n,\, j < i \leq h(j)\},
\end{equation}
where $(i,j)$ denotes the transposition interchanging $i$ and $j$.
We note that the image of $\pi^*(t_i) \in \pi^*(H^*(BT))\subset H^*_T(X(h))$ by the homomorphism in \eqref{eq:restriction_hom} is the constant function $t_i\in {\rm Map}(\Sn,\Z[t_1,\dots,t_n])$. 

Guillemin and Zara \cite{GZ} assigned a labeled graph to a GKM manifold.
The labeled graph of $X(h)$ is as follows.
The vertex set is the fixed point set $X(h)^T=\mathfrak S_n$.
There is an edge between vertices $w$ and $v$ if and only if $v = w (i,j)$ for some $j \leq i \leq h(j)$, and the edge between $w$ and $w(i,j)$ is labeled by $t_{w(i)} - t_{w(j)}$ up to sign.

\begin{example} \label{exam:GKM_graph}
Let $n=3$. For $h=(2,3,3)$ and $h'=(3,3,3)$, the labeled graphs of $X(h)$ and $X(h')$ are drawn in Figure $\ref{pic:GKM graphs}$, where we use the one-line notation for each vertex.

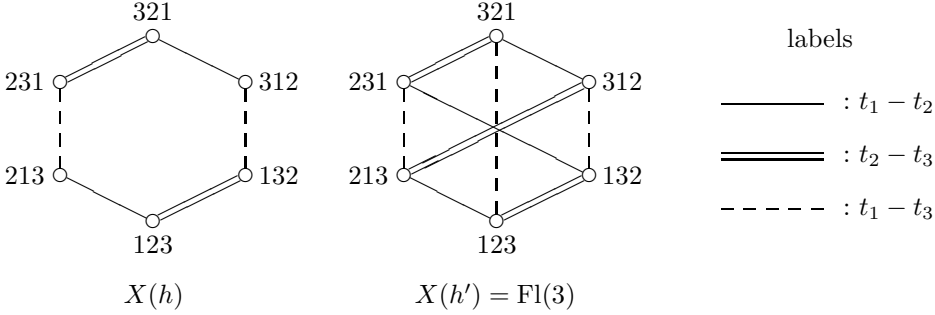
\begin{figure}[H]
\begin{center}
\begin{picture}(400,90)
\put(60,14){\circle{5}}
\put(60,84){\circle{5}}
\put(95,31.5){\circle{5}}
\put(95,66.5){\circle{5}}
\put(25,31.5){\circle{5}}
\put(25,66.5){\circle{5}}

\put(57.5,15.25){\line(-2,1){30}}
\put(92.5,67.75){\line(-2,1){30}}
\put(62,16.25){\line(2,1){30}}
\put(63,14.25){\line(2,1){30}}
\put(58,81.75){\line(-2,-1){30}}
\put(57,83.75){\line(-2,-1){30}}

\multiput(25,34)(0,8.5){4}{\line(0,1){4.5}}
\multiput(95,34)(0,8.5){4}{\line(0,1){4.5}}

\put(52.5,2){$123$}
\put(52.5,90){$321$}
\put(100,28){$132$}
\put(100,63.5){$312$}
\put(4,28){$213$}
\put(4,63.5){$231$}

\put(49,-17){$X(h)$}

\put(190,14){\circle{5}}
\put(190,84){\circle{5}}
\put(225,31.5){\circle{5}}
\put(225,66.5){\circle{5}}
\put(155,31.5){\circle{5}}
\put(155,66.5){\circle{5}}

\put(187.5,15.25){\line(-2,1){30}}
\put(222.5,67.75){\line(-2,1){30}}
\put(222.5,32.75){\line(-2,1){65}}
\put(192,16.25){\line(2,1){30}}
\put(193,14.25){\line(2,1){30}}
\put(188,81.75){\line(-2,-1){30}}
\put(187,83.75){\line(-2,-1){30}}
\put(223,64.25){\line(-2,-1){65}}
\put(222,66.25){\line(-2,-1){65}}

\multiput(155,34)(0,8.5){4}{\line(0,1){4.5}}
\multiput(225,34)(0,8.5){4}{\line(0,1){4.5}}
\multiput(190,17)(0,8.5){8}{\line(0,1){4.5}}

\put(182.5,2){$123$}
\put(182.5,90){$321$}
\put(230,28){$132$}
\put(230,63.5){$312$}
\put(134,28){$213$}
\put(134,63.5){$231$}

\put(159,-17){$X(h') = \Flag(3)$}

\put(300,80){{\rm labels}}
\put(275,58.5){\line(1,0){38.5}}
\put(320,56.5){$\colon t_1-t_2$}
\put(275,39.6){\line(1,0){38.5}}
\put(275,37.4){\line(1,0){38.5}}
\put(320,36.5){$\colon t_2-t_3$}
\multiput(275,18.5)(8.5,0){5}{\line(1,0){4.5}}
\put(320,16.5){$\colon t_1-t_3$}
\end{picture}
\end{center}
\vspace{15pt}
\caption{The labeled graphs of $X(h)$ and $X(h')$}
\label{pic:GKM graphs}
\end{figure}
\end{example}

In general, labeled graphs and their graph cohomologies are defined as follows.
\begin{definition}
Let $R$ be a ring.
A \textit{labeled graph} $(\Gamma,\alpha)$ consists of
a graph $\Gamma = (V,E)$ and a labeling $\alpha \colon E \to R$.
The \textit{graph cohomology} of a labeled graph $(\Gamma,\alpha)$ is defined as 
\[
	H^*(\Gamma,\alpha)
	= \{ f \in \mathrm{Map}(V, R) \mid f(w) - f(v) \in (\alpha(e)) \text{ for } \forall e = wv \in E \}.
\]
The graph cohomology $H^*(\Gamma, \alpha)$ is a subring of $\mathrm{Map}(V, R)$
with the coordinate-wise sum and multiplication. 
Note that we may ignore the signs of the labels $\alpha(e)$ since $(\alpha(e))=(-\alpha(e))$. 
\end{definition}
The observation above shows that the graph cohomology of the labeled graph of $X(h)$ coincides with $H^*_T(X(h))$. 

Sending $t_i$ to $t_{\sigma(i)}$ for $\sigma\in \mathfrak{S}_n$ and $i\in [n]$ induces an action of $\mathfrak{S}_n$ on $\Z[t_1,\dots,t_n]$. Then, the module ${\rm Map}(\mathfrak{S}_n,\Z[t_1,\dots,t_n])$ becomes an $\mathfrak{S}_n$-module under what is called the dot action defined by 
\[
	(\sigma\cdot f)(w):=\sigma(f(\sigma^{-1}w)).
\] 
As easily checked, the graph cohomology of $X(h)$ is invariant under the dot action and $H^*_T(X(h))$ becomes a module over $\mathfrak{S}_n$.
Moreover, since the action of $\mathfrak{S}_n$ preserves the ideal $(\pi^*(t_1),\dots,\pi^*(t_n))$,
the action descends to $H^*(X(h))$ and $H^*(X(h))$ also becomes an module over $\mathfrak{S}_n$. 

Obviously, constant functions in ${\rm Map}(\mathfrak{S}_n,\Z[t_1,\dots,t_n])$ satisfy the condition in \eqref{eq:GKM_condition_for_X(h)}.
They are elements corresponding to $\pi^*(H^*(BT))\subset H^*_T(X(h))$.
Below are three types of elements $x_i$, $y_{j,k}$, and $\tau_A$ in ${\rm Map}(\mathfrak{S}_n,\Z[t_1,\dots,t_n])$ which satisfy the condition in \eqref{eq:GKM_condition_for_X(h)},
so they are in $H^*_T(X(h))$.
Let 
\begin{equation} \label{eq:JLh}
\begin{split}
	\J(h):&=\{ j\in [n-1]\mid h(j-1)=h(j)=j+1\}\\
	\L(h):&=\{ j\in [n-1] \mid h(j-1)=j\text{ and }h(j)=j+1\}
\end{split}
\end{equation}
where we understand $h(0)=1$. 
\begin{definition} \label{defi:xytau}
\begin{enumerate}
	\item For $i\in [n]$, $x_i(w):=t_{w(i)}$.
	\item For $j\in [n-1]$ with $j\in \J(h)$ and $k\in [n]$,
	\[
		y_{j,k}(w):=
		\begin{cases}
			t_k-t_{w(j+1)} &(\text{if } k\in \{w(1),\dots,w(j)\})\\
			0 &(\text{otherwise}).
		\end{cases}
	\]
	\item For $A\subset [n]$ with $|A|\in \L(h)$ 
	\begin{equation*} \label{eq:tauA}
		\tau_A(w):=
		\begin{cases}
			t_{w(|A|)}-t_{w(|A|+1)} &(\text{if }\{w(1),\dots,w(|A|)\}=A)\\
			0 &(\text{otherwise}).
		\end{cases}
	\end{equation*}
\end{enumerate}
\end{definition}

The cohomological degrees of the elements $x_k, y_{j,k}, \tau_A$ are two. 
One can easily check that the dot actions of $\sigma\in \mathfrak{S}_n$ on these elements are given as follows:
\begin{equation} \label{eq:dot_action_on_xytau}
	\sigma\cdot x_k=x_k,\quad \sigma\cdot y_{j,k}=y_{j,\sigma(k)},\quad
	\sigma\cdot\tau_A=\tau_{\sigma(A)}.
\end{equation}

\begin{remark} \label{rema:meaning_of_x}
Here is a geometrical meaning of $x_k$'s (regarded as elements in $H^2(X(h))$ through the isomorphism \eqref{eq:quotient_of_equivariant_cohomology}). There is a nested sequence of tautological vector bundles over the flag variety $\Fl(n)$:
\[
	\mathcal{F}_0\subset \mathcal{F}_1\subset \mathcal{F}_2\subset \cdots \subset\mathcal{F}_n=\Fl(n)\times \C^n
\]
where 
\[
	\mathcal{F}_k:=\{(V_\bullet,v)\in \Fl(n)\times \C^n\mid v\in V_k\}\qquad\text{and}\qquad
V_\bullet=(\{0\}=V_0\subset V_1\subset V_2\subset \cdots \subset V_n=\C^n).
\]
Then $x_k$ $(k\in [n])$ is the image of the first Chern class of the quotient line bundle $\mathcal{F}_k/\mathcal{F}_{k-1}$ over $\Fl(n)$ by the homomorphism 
\[
	\iota^*\colon H^*(\Fl(n))\to H^*(X(h))
\]
induced from the inclusion map $\iota\colon X(h)\to \Fl(n)$.
The dot action on $H^*(\Fl(n))$ is trivial, so the image of $\iota^*$ must be contained in the ring of invariants $H^*(X(h))^{\mathfrak{S}_n}$.
In fact, it follows from \cite[Theorems A and B]{AHHM} that the image of $\iota^*$ agrees with $H^*(X(h))^{\mathfrak{S}_n}$ when tensoring with $\Q$ and 
\begin{equation} \label{eq:image_iota*}
	H^*(X(h))^{\mathfrak{S}_n}\otimes\Q=\Q[x_1,\dots,x_n]/(f_{h(1),1},\dots,f_{h(n),n})
\end{equation}
where 
\begin{equation} \label{eq:fij}
	f_{h(j),j}=\sum_{k=1}^j\left(x_k\prod_{\ell=j+1}^{h(j)}(x_k-x_\ell)\right).
\end{equation}
In particular, the Hilbert series of $H^*(X(h))^{\mathfrak{S}_n}$ is given by 
\begin{equation} \label{eq:ring_of_invariants}
	\Hilb(H^*(X(h))^{\mathfrak{S}_n},\sqrt{q})= \prod_{j=1}^{n-1} [h(j)-j]_q
\end{equation}
where the Hilbert series of a graded algebra $\mathcal{A}=\sum_{r=0}^\infty \mathcal{A}^{r}$ over $\Z$ is defined as $$\Hilb(\mathcal{A},q):=\sum_{r=0}^\infty ({\rm rank}_\Z \mathcal{A}^{r})q^r.$$
\end{remark}

Through the isomorphism \eqref{eq:quotient_of_equivariant_cohomology}, the elements $x_k, y_{j,k}, \tau_A$ determine elements in $H^2(X(h))$, denoted by the same notation. 

\begin{theorem}[{\cite[Theorem 5.1]{AMS1}}]
\label{theo:AMS1}
The elements 
\[
	\{x_k, y_{j,k}, \tau_A\mid k\in [n],\ j\in \J(h)\backslash\{n-1\},\ A\subset [n] \text{ with } |A|\in \L(h)\backslash\{n-1\}\}
\]
generate $H^2(X(h))$ with relations 
\begin{enumerate}
	\item $\sum_{k=1}^nx_k=0$,
	\item $\sum_{k=1}^ny_{j,k}=(x_1+\dots+x_j)-jx_{j+1}$ for $j\in \J(h)\backslash\{n-1\}$,
	\item $\sum_{|A|=j}\tau_A=x_j-x_{j+1}$ for $j\in \L(h)\backslash\{n-1\}$.
\end{enumerate}
\end{theorem}

\begin{remark}[see Subsection 6.2 in \cite{AMS1} for more details] \label{rema:y*}
The element $y_{j,k}$ is defined by looking at the $j$-th column of the configuration associated to the Hessenberg function $h$. Similarly, one can define an element $y^*_{i,k}$ of $H^*_T(\Hess(S,h))$ by looking at the $i$-th row of the configuration as follows.
For $i\in [n]$, we define
\[
	h^*(i):=\min\{j\in [n]\mid h(j)\ge i\},
\]
so that the shaded boxes in the $i$-th row and under the diagonal in the configuration associated to $h$ are at positions $(i,\ell)$ $(h^*(i)\le \ell <i)$. When $h^*(i)=i-1$, we define
\begin{equation} \label{eq:y*ik}
y^*_{i,k}(w):=
\begin{cases}
t_k-t_{w(i-1)}\quad&(k\in \{w(i),\dots,w(n)\})\\
0\quad&(\text{otherwise}).
\end{cases}
\end{equation}
One can see that $y^*_{i.k}$ is in $H^2_T(\Hess(S,h))$ and we may replace $y_{j,k}$'s for $j\in\J(h)\backslash\{n-1\}$ in the generating set in Theorem \ref{theo:AMS1} by $y^*_{i,k}$'s for $i\ge 3$ such that $h^*(i)=h^*(i+1)=i-1$. 
\end{remark}

\begin{example} 
When $h=(4,4,4,5,6,7,11,11,11,11)$ in Figure \ref{pic:stair-shape} (i.e.\ $(a,b)=(3,7)$), we have
\[
	\J(h)=\{3,10\},\quad \L(h)=\{4,5,6\},
\]
so Theorem \ref{theo:AMS1} says that $H^2(X(h))$ is generated by 
\[
	\text{$x_k$ $(k\in [11])$,\quad
	$y_{3,k}$ $(k\in [11])$,\quad
	$\tau_A$ for $A\subset [11]$ with $|A|=4,5$ or $6$. }
\]
Moreover, it follows from Remark \ref{rema:y*} that $y_{3,k}$ above may be replaced by $y^*_{8,k}$. 
\end{example}

\section{Necessity}
In this section, we study a necessary condition on $h$ for $H^*(X( h))$ to be generated in degree $2$ as a ring.

\subsection{Moment maps}
Let $\mu \colon \Flag(n)\to \R^n$ be the standard moment map on the flag variety $\Flag(n)$.
Its image is the permutohedron $\Pi_n$ obtained as the convex hull of the orbits of $(1,2, \dots,n)$ by permuting its coordinates.
Indeed, if $e_w$ $(w\in \mathfrak{S}_n)$ denotes the permutation flag associated with $w$, then we have
\[
	\mu(e_w)=(w^{-1}(1), \dots, w^{-1}(n))\in \R^n
\]
(see \cite[Lemma 3.1]{LMP2021} for example).
Let 
\begin{equation} \label{eq:Snr}
	\mathfrak{S}_n^r := \{ w \in \mathfrak{S}_n \mid w(r) = n\}.
\end{equation}
Then $\mu(\mathfrak{S}_n^r)$ is the set of all vertices of $\Pi_n$ whose $n$-th coordinate is $r$.
Therefore the projection
\[
	\pi_n\colon \Pi_n \to \R, \qquad \pi_n(x_1, \dots, x_n)=x_n
\]
on the $n$-th coordinate takes minimum on $\mathfrak{S}_n^1$ and maximum on $\mathfrak{S}_n^n$.
The composition of $\mu$ and $\pi_n$
\begin{equation} \label{eq:Morse_Bott}
	f :=\pi_n\circ \mu\colon \Flag(n)\to \R
\end{equation}
is the moment map induced from the following $S^1$-action on $\C^n$
\begin{equation} \label{eq:S1-action}
	(z_1, \dots,z_n)\to (z_1, \dots,z_{n-1},gz_n)\quad (g\in S^1\subset \C),
\end{equation}
and it is a Morse-Bott function.

Let $h^j$ be the Hessenberg function obtained by removing all the boxes in the $j$-th row and all the boxes in the $j$-th column from its configuration (see Figure \ref{pic:h^j}). To be precise, $h^j$ is given as follows.
\[
	h^j(i) =
	\begin{cases}
		h(i) & (i < j,\ h(i)<j)\\
		h(i)-1 & (i < j,\ h(i) \geq j)\\
		h(i+1) -1 & (i\geq j)
	\end{cases}
\]
\begin{figure}[H]
\begin{center}
\setlength{\unitlength}{5mm}
\begin{picture}(24,8)(-3,-1)
\linethickness{0.3mm}
\put(-4,3.3){$j$-th row $\rightarrow$}
\put(1.35,5.5){$\downarrow$}
\put(0.5,6.3){$j$-th column}
\put(2.3,-1.2){$h$}
\multiput(0,4.2)(0,-1){3}{\colorbox{gray7}{\phantom{\vrule width 3mm height 3mm}}}
\multiput(1,4.2)(0,-1){3}{\colorbox{gray7}{\phantom{\vrule width 3mm height 3mm}}}
\multiput(2,4.2)(0,-1){4}{\colorbox{gray7}{\phantom{\vrule width 3mm height 3mm}}}
\multiput(3,4.2)(0,-1){5}{\colorbox{gray7}{\phantom{\vrule width 3mm height 3mm}}}
\multiput(4,4.2)(0,-1){5}{\colorbox{gray7}{\phantom{\vrule width 3mm height 3mm}}}
\multiput(1,0)(1,0){4}{\line(0,1){5}}
\multiput(0,1)(0,1){4}{\line(1,0){5}}
\put(0,0){\framebox(5,5){}}
\put(6.75,2.3){$\leadsto$}
\put(5.9,1.7){remove}
\put(10.15,4.3){$\leftarrow$}
\put(10.15,3.3){$\nwarrow$}
\put(9.3,3.3){$\uparrow$}
\put(9,4.2){\colorbox{gray7}{\phantom{\vrule width 3mm height 3mm}}}
\put(11,4.2){\colorbox{gray7}{\phantom{\vrule width 3mm height 3mm}}}
\put(12,4.2){\colorbox{gray7}{\phantom{\vrule width 3mm height 3mm}}}
\put(13,4.2){\colorbox{gray7}{\phantom{\vrule width 3mm height 3mm}}}
\multiput(9,2.2)(0,-2){1}{\colorbox{gray7}{\phantom{\vrule width 3mm height 3mm}}}
\multiput(11,2.2)(0,-1){2}{\colorbox{gray7}{\phantom{\vrule width 3mm height 3mm}}}
\multiput(12,2.2)(0,-1){3}{\colorbox{gray7}{\phantom{\vrule width 3mm height 3mm}}}
\multiput(13,2.2)(0,-1){3}{\colorbox{gray7}{\phantom{\vrule width 3mm height 3mm}}}
\multiput(12,0)(1,0){2}{\line(0,1){3}}
\multiput(12,4)(1,0){2}{\line(0,1){1}}
\multiput(9,1)(0,1){2}{\line(1,0){1}}
\multiput(11,1)(0,1){2}{\line(1,0){3}}
\put(9,4){\framebox(1,1){}}
\put(9,0){\framebox(1,3){}}
\put(11,4){\framebox(3,1){}}
\put(11,0){\framebox(3,3){}}
\put(15.75,2.3){$\leadsto$}
\put(19.8,-1.2){$h^j$}
\multiput(18,4.2)(0,-1){2}{\colorbox{gray7}{\phantom{\vrule width 3mm height 3mm}}}
\multiput(19,4.2)(0,-1){3}{\colorbox{gray7}{\phantom{\vrule width 3mm height 3mm}}}
\multiput(20,4.2)(0,-1){4}{\colorbox{gray7}{\phantom{\vrule width 3mm height 3mm}}}
\multiput(21,4.2)(0,-1){4}{\colorbox{gray7}{\phantom{\vrule width 3mm height 3mm}}}
\multiput(19,1)(1,0){3}{\line(0,1){4}}
\multiput(18,2)(0,1){3}{\line(1,0){4}}
\put(18,1){\framebox(4,4)}
\end{picture}
\end{center}
\caption{The configuration corresponding to $h^j$.}
\label{pic:h^j}
\end{figure}
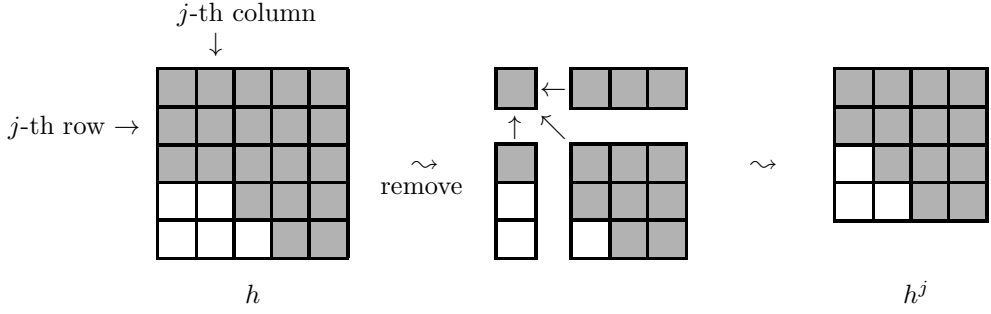

The following is a key lemma in our argument. 

\begin{lemma} \label{lemm:surjection}
The restriction maps
\begin{equation*} \label{eq:h1hn}
	H^*(X( h);\Q)\to H^*(X( h^1);\Q),\quad H^*(X(h);\Q)\to H^*(X(h^n);\Q)
\end{equation*} 
are surjective.
\end{lemma}

\begin{proof}
Let $f_h$ be the map $f$ in \eqref{eq:Morse_Bott} restricted to $X( h)$, which is also a Morse-Bott function.
The inverse image of the minimum value under $f_h$ is $X(h^1)$, so it follows from \cite[Lemma 3.1]{to-we03} that the restriction map 
\begin{equation} \label{eq:S1-surjection}
	H^*_{S^1}(X(h);\Q)\to H^*_{S^1}(X(h^1);\Q)
\end{equation}
is surjective,
where the $S^1$-action on $X(h)$ is the induced one from the $S^1$-action defined in \eqref{eq:S1-action}.
Since the $S^1$-action on $X(h^1)$ is trivial, we have $H^*_{S^1}(X(h^1);\Q)=H^*(BS^1;\Q)\otimes H^*(X(h^1);\Q)$ and hence the forgetful map $H^*_{S^1}(X(h^1);\Q)\to H^*(X(h^1);\Q)$ is surjective.   
Therefore, the surjectivity of \eqref{eq:S1-surjection} implies the surjectivity of the restriction map
\[
	H^*(X(h);\Q)\to H^*(X(h^1);\Q)
\]
in ordinary cohomology.  The same argument applied to $-f_h$ proves the statement for $X(h^n)$.  
\end{proof}


\begin{remark}
The surjectivity of the above restriction maps (even with $\Z$ coefficients) can also be verified by GKM theory as follows.
Recall that the inclusion of the fixed point set induces an injective homomorphism
$H^*_T(X(h)) \to H^*_T(X(h)^T) \cong \mathrm{Map}(\mathfrak{S}_n, H^*(BT))$.
The equivariant cohomology $H^*_T(X(h))$ has an $H^*(BT)$-module basis
$\{\sigma_{w,h}\mid w \in \mathfrak{S}_n\}$ (see \cite[Definition 2.9 and Proposition 2.11]{CHL21}).
It corresponds to a natural paving and then it is a `flow-up basis.'
Note that any element of $\mathfrak{S}_n^n = \mathfrak{S}_{n-1}$ is not greater than
any element of $\mathfrak{S}_{n} \setminus \mathfrak{S}_n^n$.
The restriction of $\{\sigma_{w,h}\mid w \in \mathfrak{S}_n^n\}$ onto $X(h^n)$,
that is, its restriction onto the fixed point set $\mathfrak{S}_n^n = X(h^n)^T$ as elements of $\mathrm{Map}(\mathfrak{S}_n, H^*(BT))$, is a flow-up basis of $H^*_T(X(h^n))$.
Hence $H^*_T(X(h))\to H^*_T(X(h^n))$ is surjective, and then $H^*(X(h))\to H^*(X(h^n))$ is also surjective.
The surjectivity of $H^*(X( h))\to H^*(X( h^1))$ can be verified by a similar argument.
\end{remark}

Given a Hessenberg function $h$, we obtain a smaller Hessenberg function by removing the first column and row or the last column and row repeatedly, i.e.\ by taking $h^1$ or $h^n$ repeatedly.
We call it a minor of $h$.  The following corollary follows from Lemma \ref{lemm:surjection}.
\begin{corollary}
Let $h'$ be a minor of $h$. If $H^*(X( h);\Q)$ is generated in degree $2$ as a ring, then so is $H^*(X( h');\Q)$.
\end{corollary}

An easy argument shows that $h$ being of the form \eqref{eq:h} can be rephrased as follows.
\begin{proposition}
\label{prop:minor}
The Hessenberg function $h$ is of the form \eqref{eq:h}
if and only if $h$ has neither
\[
	(\alpha,\beta, \dots, \beta),\ 
	(\beta-1, \ldots, \beta-1, \underbrace{\beta, \ldots, \beta}_{\alpha}) \ \text{for } 3\le \alpha<\beta,\ 
	\text{nor} \ (2,\gamma-1, \dots, \gamma-1, \gamma, \gamma)\ \text{for }\gamma\ge 5
\]
as its minor.
\end{proposition}
Recall that if $h^\dagger$ denotes the Hessenberg function obtained by flipping the configuration of $h$ along the anti-diagonal, then $X(h^\dagger)\cong X(h)$ as varieties. Therefore 
\[
	X((\alpha,\beta, \dots, \beta)) \cong
	X((\beta-1, \ldots, \beta-1, \underbrace{\beta, \ldots, \beta}_{\alpha})).
\]
Here, we know that $H^*(X((\alpha,\beta, \dots, \beta));\Q)$ is not generated in degree $2$ for $3\le \alpha<\beta$ by \cite[Theorem 4.3]{AHM}.
Thus, it suffices to treat the last case in Proposition \ref{prop:minor}, which we shall discuss in the next subsection. 

\subsection{The case $h=(2,n-1,\dots,n-1,n,n)$}
In this subsection we prove the following proposition.
\begin{proposition}
\label{prop:not-generate}
$H^*(X(h);\Q)$ is not generated in degree $2$ when $h=(2,n-1, \dots,n-1,n,n)$ for $n\ge 5$. 
\end{proposition}

Some computation is involved in the proof of this proposition but the idea of the proof is simple. We compute the Poincar\'e polynomial of $X(h)$ using Theorem \ref{thm:MPS}(4).
On the other hand, using explicit generators of $H^2(X(h))$ by \cite{AMS1}, we compute an upper bound of the Hilbert series of the subring of $H^*(X(h))$ generated by $H^2(X(h))$.
Then it turns out that the latter is strictly smaller than the former at a certain degree. 

\subsubsection{Poincar\'e polynomial of $X(h)$}
The following proposition, which easily follows from Theorem \ref{thm:MPS}(4), enables us to compute the Poincar\'e polynomial of $X(h)$ inductively. 

\begin{proposition}[{\cite[Proposition 3.1]{AMS1}}]
\begin{equation} \label{eq:Poincare_inductive}
	\Poin(X( h),\sqrt{q})=\sum_{j=1}^nq^{h(j)-j}\Poin(X( h^j),\sqrt{q}).
\end{equation}
\end{proposition}

Using the proposition above, the Poincar\'e polynomial of $X(h)$ is explicitly computed as follows when $h=(h(1),n,\dots,n)$. 

\begin{proposition}[\cite{AHM}] 
When $h = (h(1),n\dots,n)$, we have 
\begin{equation} \label{eq:Poincare_h1}
	\Poin(X( h),\sqrt{q})=[h(1)]_q[n-1]_q!+(n-1)q^{h(1)-1}[n-h(1)]_q[n-2]_q!,
\end{equation}
where 
\[
	[m]_q = \frac{1-q^m}{1-q},\quad
	[m]_q! = [1]_q[2]_q\cdots [m]_q = \prod_{j=1}^{m}\frac{1-q^j}{1-q}.
\]
\end{proposition}

Now, let $h=(2,n-1, \dots,n-1,n,n)$ and set
\[
	P_n(q):=\Poin(X( h),\sqrt{q}).
\]

\begin{lemma} \label{lem:recurrence}
For $n\ge 5$, the following recurrence formula holds
\begin{align*}
	P_n(q)=\ &(1+q)^2[n-2]_q!+(n-2)(q+q^2)[n-3]_q[n-3]_q!\\
	&+(n-1)(q+q^{n-3})\left\{(1+q)[n-3]_q!+(n-3)q[n-4]_q[n-4]_q!\right\}\\
	&+(q+q^2+\cdots+q^{n-4})P_{n-1}(q).
\end{align*}
\end{lemma}

\begin{proof}
Let $F_n(q)$ denote the right-hand side of \eqref{eq:Poincare_h1} with $h(1)=2$, that is,
\begin{equation} \label{eq:Fn}
	F_n(q):=(1+q)[n-1]_q!+(n-1)q[n-2]_q[n-2]_q!.
\end{equation}
Then we have
\begin{align*}
	&\Poin(X( h^1),\sqrt{q})=\Poin(X( h^n),\sqrt{q})=F_{n-1}(q)\\
	&\Poin(X( h^2),\sqrt{q})=\Poin(X( h^{n-1}),\sqrt{q})=(n-1)F_{n-2}(q)\\
	&\Poin(X( h^j),\sqrt{q})=P_{n-1}(q) \quad (3\le j\le n-2),
\end{align*}
where we note that $X(h^2)$ consists of $n-1$ copies of $\Fl(n-2)$. 
Hence, by \eqref{eq:Poincare_inductive}, we have
\begin{align*}
	P_n(q)=\
	&qF_{n-1}(q)+(n-1)q^{n-3}F_{n-2}(q)+(q^{n-4}+\cdots+q)P_{n-1}(q)\\
	&+(n-1)qF_{n-2}(q)+F_{n-1}(q)\\
	=\
	&(1+q)F_{n-1}(q)+(n-1)(q+q^{n-3})F_{n-2}(q)+(q+\cdots+q^{n-4})P_{n-1}(q).
\end{align*}
Combining this equation with \eqref{eq:Fn}, we obtain the desired equation.
\end{proof}

\begin{lemma} \label{lem:Pn=Qn}
For $n\ge 4$, let
\[
	Q_n(q)=(1+2nq+n(n-1)q^2)[n-2]_q!+\frac{n(n-3)}{2}q^{n-3}.
\]
Then we have
\[
	P_n(q)\equiv Q_n(q) \mod{(q^{n-2})}.
\]
In other words, $P_n(q)$ and $Q_n(q)$ coincide up to degree ${n-3}$.
\end{lemma}

\begin{proof}
We prove the lemma by induction on $n$. When $n=4$, we have
\[
	P_4(q)=1+11q+11q^2+q^3,\quad Q_4(q)=1+11q+20q^2+12q^3,
\]
and the lemma is true for $n=4$.

Let $n$ be given and suppose that the lemma is true for $n-1$, that is,
\begin{equation} \label{eq:Pn-1=Qn-1}
	P_{n-1}(q)\equiv Q_{n-1}(q)\mod{(q^{n-3})}.
\end{equation}
Hereafter, in this proof, all congruences will be taken modulo $q^{n-2}$ unless otherwise stated. Since we have
\begin{align*}
	(q+q^2)[n-3]_q[n-3]_q!&\equiv (q+q^2)[n-2]_q!\\
	q^2[n-4]_q[n-4]_q!&\equiv q^2[n-3]_q! ,
\end{align*}
the recurrence formula in Lemma \ref{lem:recurrence} reduces to the following congruence relation: 
\begin{equation} \label{eq:Pn_mod}
\begin{split}
	P_n(q)\equiv
	&\ (1+nq+(n-1)q^2)[n-2]_q!+(n-1)(q+(n-2)q^2)[n-3]_q!\\
	&+(n-1)q^{n-3}+(q+\cdots+q^{n-4})P_{n-1}(q).
\end{split}
\end{equation}
It follows from \eqref{eq:Pn-1=Qn-1} and the definition of $Q_{n}$ that the sum of the last two terms above becomes as follows.
\begin{align*}
	&\ (n-1)q^{n-3}+(q+\cdots+q^{n-4})P_{n-1}(q)\\
	\equiv&\ (n-1)q^{n-3}+\left(1+(2n-2)q+(n-1)(n-2)q^2\right)(q+\cdots+q^{n-4})[n-3]_q! +\frac{(n-1)(n-4)}{2}q^{n-3}\\
	= &\ \left\{1-q+(n-1)q(1-q)+nq+(n-1)^2q^2\right\}(q+\cdots+q^{n-4})[n-3]_q! +\frac{(n-1)(n-2)}{2}q^{n-3}\\
	\equiv &\ \left\{q-q^{n-3}+(n-1)q^2+(nq+(n-1)^2q^2)(q+\cdots+q^{n-4})\right\}[n-3]_q! +\frac{(n-1)(n-2)}{2}q^{n-3}\\
	\equiv &\ \left\{q+(n-1)q^2+(nq+(n-1)^2q^2)(q+\cdots+q^{n-4})\right\}[n-3]_q! +\frac{n(n-3)}{2}q^{n-3}\\
\end{align*}
By substituting it to \eqref{eq:Pn_mod}, we obtain
\begin{align*}
	P_n(q)\equiv &\ (1+nq+(n-1)q^2)[n-2]_q!\\
	&+\left\{(nq+(n-1)^2q^2)+(nq+(n-1)^2q^2)(q+\cdots+q^{n-4})\right\}[n-3]_q! +\frac{n(n-3)}{2}q^{n-3}\\
	\equiv &\ (1+nq+(n-1)q^2)[n-2]_q! +(nq+(n-1)^2q^2)[n-2]_q!+\frac{n(n-3)}{2}q^{n-3}\\
	= &\ (1+2nq+n(n-1)q^2)[n-2]_q!+\frac{n(n-3)}{2}q^{n-3}\\
	=&\ Q_n(q).
\end{align*}
This completes the induction step and the lemma has been proved.
\end{proof}

\subsubsection{Hilbert series of the subring generated by $H^2(X(h))$}
When $h=(2,n-1, \dots,n-1,n,n)$ for $n\ge 5$, we first observe $H^2(X(h))$. By \eqref{eq:JLh}, we have 
\[
	\J(h)=\{n-2\}, \quad \L(h)=\{1, n-1\}.
\]
Therefore, it follows from Theorem \ref{theo:AMS1} that $H^2(X(h))$ is generated by the following elements 
\begin{equation} \label{eq:generating_set}
	x_k,\qquad y_k := y_{n-2,k},\qquad \tau_k := \tau_{\{k\}}\qquad (k\in [n]), 
\end{equation}
where 
\begin{equation} \label{eq:tau_k}
\begin{split}
	x_k(w)&=t_{w(k)},\\
	y_k(w)&=y_{n-2,k}(w)=
	\begin{cases}
		t_k-t_{w(n-1)} & (\text{if } k\in \{w(1),\dots,w(n-2)\})\\
		0 & \text{(otherwise)},
	\end{cases}\\
	\tau_k(w) &=\tau_{\{k\}}(w)=
	\begin{cases}
		t_{w(1)}-t_{w(2)} & (\text{if } k=w(1))\\
		0 &(\text{otherwise})
	\end{cases}\\
\end{split}
\end{equation}
for $w\in \mathfrak{S}_n$ by Definition \ref{defi:xytau}, and 
\begin{equation} \label{eq:sum_yk}
	\sum_{k=1}^ny_k=x_1+\cdots+x_{n-2}-(n-2)x_{n-1},\qquad \sum_{k=1}^n\tau_k=x_1-x_2
\end{equation}
by Theorem \ref{theo:AMS1}. We also have 
\[
	\sigma\cdot x_k=x_k,\qquad 
	\sigma\cdot y_k=y_{\sigma(k)},\qquad 
	\sigma\cdot \tau_k=\tau_{\sigma(k)} 
\]
for $\sigma\in \mathfrak{S}_n$ by \eqref{eq:dot_action_on_xytau}. 

To make the following argument clearer, we introduce elements $\rho_k$ for $k\in [n]$ defined by
\begin{equation} \label{eq:rho_k}
	\rho_k(w) :=
	\begin{cases}
		t_{w(n-1)}-t_{w(n)} & (\text{if }k=w(n))\\
		0 &(\text{otherwise}).
	\end{cases}
\end{equation}
Similarly to $\tau_k$, the $\rho_k$ satisfies the condition \eqref{eq:GKM_condition_for_X(h)} so that it defines an element of $H_T^2(X(h))$ and $H^2(X(h))$ and 
\begin{equation} \label{eq:sum_rhok}
	\sum_{k=1}^n\rho_k=x_{n-1}-x_n,\qquad
	\sigma\cdot \rho_k=\rho_{\sigma(k)}\quad\text{for $\sigma\in\mathfrak{S}_n$}.
\end{equation}
An elementary check shows that 
\[
	(y_k-y_\ell)(w)-(\rho_k-\rho_\ell)(w)=t_k-t_\ell\qquad (k,\ell\in [n],\ w\in \mathfrak{S}_n)
\]
and hence $y_k-y_\ell=\rho_k-\rho_\ell$ in $H^2(X(h))$.
Moreover, $\sum_{k=1}^ny_k$ and $\sum_{k=1}^n\rho_k$ are both linear polynomials in $x_i$'s by \eqref{eq:sum_yk} and \eqref{eq:sum_rhok},
so we may replace $y_k$'s in the generating set \eqref{eq:generating_set} by $\rho_k$'s.
Namely $H^2(X(h))$ is generated by
\[
	x_k,\qquad \tau_k,\qquad \rho_k\qquad (k\in [n])
\]
with relations 
\begin{equation} \label{eq:relation_xtaurho}
	\sum_{k=1}^nx_k=0, \qquad
	\sum_{k=1}^n\tau_k=x_1-x_2,\qquad
	\sum_{k=1}^n\rho_k=x_{n-1}-x_n,
\end{equation}
and the actions of $\sigma\in \mathfrak{S}_n$ on those generators are given by 
\begin{equation} \label{eq:new_generating_set}
	\sigma\cdot x_k=x_k,\qquad
	\sigma\cdot \tau_k=\tau_{\sigma(k)},\qquad
	\sigma\cdot \rho_k=\rho_{\sigma(k)}.
\end{equation}

Our purpose is to find a sharp upper bound of the Hilbert series of the subring $\mathcal{R}(h)$ of $H^*(X(h))$ generated by $H^2(X(h))$.
Let ${A}(h)$ be the subring of $H^*(X(h))$ generated by $x_k$'s and we regard $\mathcal{R}(h)$ as a module over ${A}(h)$.
It follows from \eqref{eq:tau_k} and \eqref{eq:rho_k} that 
\begin{equation*} \label{eq:product_relation}
	\tau_k\tau_\ell =
	\begin{cases}
		(x_1-x_2)\tau_k & (k=\ell)\\
		0 &(k\not=\ell),
	\end{cases}
	\qquad
	\rho_k\rho_\ell =
	\begin{cases}
		(x_{n-1}-x_n)\rho_k & (k=\ell)\\
		0 &(k\not=\ell),
	\end{cases}
	\qquad
	\tau_k\rho_k=0.
\end{equation*}
Therefore, $\mathcal{R}(h)$ is generated by $1$, $\tau_k$, $\rho_k$ $(k\in [n])$, and $\tau_i\rho_j$ $(i\not=j\in [n])$ as a module over ${A}(h)$. The subring $A(h)$ itself is a submodule of $\mathcal{R}(h)$ over $A(h)$. We consider three other submodules of $\mathcal{R}(h)$ over $A(h)$:
\begin{equation} \label{eq:BCD}
\begin{split}
	B(h):=&\{\sum_{k=1}^nb_k\tau_k\mid b_k\in A(h),\ \sum_{k=1}^nb_k=0\},\\
	C(h):=&\{\sum_{k=1}^nc_k\rho_k\mid c_k\in A(h),\ \sum_{k=1}^nc_k=0\},\\
	D(h):=&\{\sum_{1\le i,j\le n}d_{ij}\tau_i\rho_j\mid d_{ij}\in A(h),\ \sum_{j=1}^nd_{ij}=0\ \text{for $i\in [n]$},\ \sum_{i=1}^nd_{ij}=0\ \text{for $j\in [n]$}\}\\
\end{split}
\end{equation}
where $d_{kk}=0$ for $k\in [n]$. Note that $A(h)\otimes\Q$ agrees with the ring of invariants $H^*(X(h);\Q)^{\mathfrak{S}_n}$ as mentioned in Remark \ref{rema:meaning_of_x}. 

\begin{lemma} \label{lemm:decomposition_H*}
$\mathcal{R}(h)$ is additively generated by $A(h), B(h), C(h)$, and $D(h)$ when tensoring with $\Q$. 
\end{lemma}

\begin{proof}
Since $H^*(X(h))$ is generated by $1$, $\tau_k$, $\rho_k$ $(k\in [n])$, and $\tau_i\rho_j$ $(i\not=j\in [n])$ as a module over $A(h)$,
it suffices to show that any element of the form
\begin{equation} \label{eq:f}
	\sum_{k=1}^n b_k\tau_k+\sum_{k=1}^nc_k\rho_k+\sum_{1\le i,j\le n}d_{ij}\tau_i\rho_j\qquad (b_k,c_k,d_{ij}\in A(h),\ d_{kk}=0)
\end{equation}
can be expressed as a sum of elements in $A(h)$, $B(h)$, $C(h)$, and $D(h)$ when tensoring with $\Q$. 

Step 1.
Set $b:=\sum_{k=1}^n b_k$ and $c:=\sum_{k=1}^nc_k$.
Since $\sum_{k=1}^n\tau_k=x_1-x_2$ and $\sum_{k=1}^n\rho_k=x_{n-1}-x_n$ by \eqref{eq:relation_xtaurho}, we have 
\[
	\sum_{k=1}^nb_k\tau_k+\sum_{k=1}^nc_k\rho_k
	=\sum_{k=1}^n\left(b_k-\frac{b}{n}\right)\tau_k+\frac{b}{n}(x_1-x_2)
	+\sum_{k=1}^n\left(c_k-\frac{c}{n}\right)\rho_k+\frac{c}{n}(x_{n-1}-x_n). 
\]
Here the two sums at the right hand side above respectively belong to $B(h)\otimes\Q$ and $C(h)\otimes\Q$, and the remaining two terms belong to $A(h)\otimes\Q$. 

Step 2.
As for the last term in \eqref{eq:f}, since $\sum_{i=1}^n\tau_i=x_1-x_2$, we have 
\begin{equation} \label{eq:sum_dij}
\begin{split}
	\sum_{1\le i,j\le n}d_{ij}\tau_i\rho_j 
	&=\sum_{j=1}^n\left(\sum_{i=1}^n\left(d_{ij}-\frac{d_j}{n}\right)\tau_i\right)\rho_j+\sum_{j=1}^n\frac{d_j}{n}(x_1-x_2)\rho_j\\
	&=\sum_{1\le i,j\le n}\tilde{d}_{ij}\tau_i\rho_j+\sum_{j=1}^n\frac{d_j}{n}(x_1-x_2)\rho_j\\
\end{split}
\end{equation}
where 
\[
	d_j:=\sum_{i=1}^nd_{ij}\quad \text{and}\quad
	\tilde{d}_{ij}:=d_{ij}-\frac{d_j}{n}.
\]
The last sum in \eqref{eq:sum_dij} is a sum of elements in $A(h)\otimes \Q$ and $C(h)\otimes\Q$ by Step 1.
We shall show that the sum $\sum_{1\le i,j\le n}\tilde{d}_{ij}\tau_i\rho_j$ in \eqref{eq:sum_dij} is a sum of elements in $A(h)\otimes\Q$, $B(h)\otimes\Q$, and $D(h)\otimes\Q$.
We note that 
\begin{equation} \label{eq:tilde_dij}
	\sum_{i=1}^n\tilde{d}_{ij}=\sum_{i=1}^n\left(d_{ij}-\frac{d_j}{n}\right)=\sum_{i=1}^nd_{ij}-d_j=0
\end{equation}
and set 
\begin{equation} \label{eq:tilde_di}
	\tilde{d}_i:=\sum_{j=1}^n\tilde{d}_{ij}. 
\end{equation}
Since $\sum_{j=1}^n\rho_j=x_{n-1}-x_n$, we have 
\begin{equation} \label{eq:tilde_dij_taurho}
	\sum_{1\le i,j\le n}\tilde{d}_{ij}\tau_i\rho_j
	=\sum_{i=1}^n\left(\sum_{j=1}^n\left(\tilde{d}_{ij}-\frac{\tilde{d}_i}{n}\right)\rho_j\right)\tau_i
	+\sum_{i=1}^n\frac{\tilde{d}_i}{n}(x_{n-1}-x_n)\tau_i.
\end{equation}
Here the second sum at the right hand side of \eqref{eq:tilde_dij_taurho} is a sum of elements in $A(h)\otimes\Q$ and $B(h)\otimes\Q$ by Step 1.
As for the coefficients $\tilde{d}_{ij}-\frac{\tilde{d}_i}{n}$ of $\tau_i\rho_j$
in the first sum at the right hand side of \eqref{eq:tilde_dij_taurho},
it follows from \eqref{eq:tilde_dij} and \eqref{eq:tilde_di} that we have 
\[
\begin{split}
	\sum_{i=1}^n\left(\tilde{d}_{ij}-\frac{\tilde{d}_i}{n}\right)
	&=\sum_{i=1}^n\tilde{d}_{ij}-\frac{1}{n}\sum_{i=1}^n\tilde{d}_i
	=-\frac{1}{n}\sum_{i=1}^n\sum_{j=1}^n\tilde{d}_{ij}
	=-\sum_{j=1}^n\left(\sum_{i=1}^n\tilde{d}_{ij}\right)
	=0,
	\\
	\sum_{j=1}^n\left(\tilde{d}_{ij}-\frac{\tilde{d}_i}{n}\right)
	&=\sum_{j=1}^n\tilde{d}_{ij}-\tilde{d}_i=0.
\end{split}
\]
Thus, the first sum at the right hand side of \eqref{eq:tilde_dij_taurho} belongs to $D(h)\otimes\Q$.
This completes the proof of the lemma. 
\end{proof}

We shall calculate upper bounds of the Hilbert series of $A(h), B(h), C(h)$, and $D(h)$.

\underbar{Hilbert series of $A(h)$}.
Since $A(h)\otimes\Q=H^*(X(h))^{\mathfrak{S}_n}\otimes\Q$ and $h=(2,n-1,\dots,n-1,n,n)$ in our case, it follows from \eqref{eq:ring_of_invariants} that 
\begin{equation} \label{eq:A(h)}
	\Hilb(A(h), \sqrt{q})= \prod_{j=1}^{n-1} [h(j)-j]_q = (1+q)^2[n-2]_q!.
\end{equation}

\medskip
\underbar{Hilbert series of $B(h)$}. 
It follows from\eqref{eq:tau_k} that $(x_1-t_k)\tau_k$ vanishes at every $w\in \mathfrak{S}_n$, so we have 
\begin{equation} \label{eq:x1tau}
	(x_1-t_k)\tau_k=0 \quad \text{in $H^*_T(X(h))$}\quad
	\text{and hence}\quad x_1\tau_k=0\quad \text{in $H^*(X( h))$}. 
\end{equation}
Therefore, $B(h)$ is indeed a module over $A(h)/(x_1)$.
Here
\[
	A(h)/(x_1)\otimes\Q= A(h^1)\otimes\Q
\]
by \eqref{eq:image_iota*} and \eqref{eq:fij}.
Since $h^1=(n-2,\dots,n-2,n-1,n-1)$, it follows from \eqref{eq:ring_of_invariants} that 
\[
	\Hilb(A(h)/(x_1), \sqrt{q})=\prod_{j=1}^{n-2}[h^1(j)-j]_q=(1+q)[n-2]_q!.
\]
Since $B(h)$ is a module over $A(h)/(x_1)$ generated by $\tau_i-\tau_{i+1}$ $(i\in [n-1])$ and the cohomological degrees of $\tau_k$'s are two,
we obtain an upper bound of $\Hilb(B(h),q)$ as follows: 
\begin{equation} \label{eq:B(h)}
\Hilb(B(h),\sqrt{q})\le (n-1)q\Hilb(A(h)/(x_1), \sqrt{q})=(n-1)(q+q^2)[n-2]_q!. 
\end{equation}
Here $\sum_{i=0}^\infty a_iq^i\le \sum_{i=0}^\infty b_iq^i$ $(a_i,b_i\in \Z)$ means that $a_i\le b_i$ for all $i$'s. 

\medskip
\underbar{Hilbert series of $C(h)$}. 
To $f\in {\rm Map}(\mathfrak{S}_n,\Z[t_1,\dots,t_n])$ we associate $f^\vee\in {\rm Map}(\mathfrak{S}_n,\Z[t_1,\dots,t_n])$ defined by 
\[
	f^\vee(w):=f(ww_0)\qquad \text{for $w\in \mathfrak{S}_n$}, 
\]
where $w_0$ denotes the longest element in $\mathfrak{S}_n$, i.e.\ 
$w_0=n\ n-1\ \cdots\ 2\ 1$ in one-line notation.
This defines an involution on ${\rm Map}(\mathfrak{S}_n,\Z[t_1,\dots,t_n])$ and one can easily check that 
\[
	x_k^\vee=x_{n-k+1},\qquad \tau_k^\vee=-\rho_k,\qquad \rho_k^\vee=-\tau_k
\] 
from \eqref{eq:tau_k} and \eqref{eq:rho_k}.
Hence the involution gives an isomorphism between $B(h)$ and $C(h)$,
and the same inequality as \eqref{eq:B(h)} holds for $C(h)$, i.e.\ 
\begin{equation} \label{eq:C(h)}
\Hilb(C(h), \sqrt{q})\le (n-1)(q+q^2)[n-2]_q!. 
\end{equation}

\medskip
\underbar{Hilbert series of $D(h)$}.
We have $x_1\tau_k=0$ by \eqref{eq:x1tau}.
Similarly we have $x_n\rho_k=0$ since $(x_1\tau_k)^\vee=-x_n\rho_k$.
(The fact $x_n\rho_k=0$ also follows from the definition \eqref{eq:tau_k} and \eqref{eq:rho_k} of $x_k$ and $\rho_k$.)
Therefore, $D(h)$ is indeed a module over $A(h)/(x_1,x_n)$. 

As mentioned in Remark \ref{rema:meaning_of_x},
$A(h)\otimes\Q=H^*(X(h))^{\mathfrak{S}_n}\otimes\Q$ and
it is the image of the restriction map $\iota^*\colon H^*(\Fl(n))\to H^*(X(h))$.
Therefore, $A(h)/(x_1,x_n)$ is the image of the restriction map from $H^*(\Fl(n-2))$ and hence 
\[
	\Hilb(A(h)/(x_1,x_n), \sqrt{q})\le [n-2]_q!.
\]
(In fact, the equality holds above.) 
There are $2n$ relations among $d_{ij}$ $(i\not=j)$ in the definition \eqref{eq:BCD} of $D(h)$,
but one relation can be obtained from the other $2n-1$ relations because $\sum_{i=1}^n\left(\sum_{j=1}^nd_{ij}\right)=\sum_{j=1}^n\left(\sum_{i=1}^nd_{ij}\right)$.
Moreover, there are $n(n-1)$ number of $d_{ij}$'s and the cohomological degree of $\tau_i\rho_j$ is four.
Thus
\begin{equation} \label{eq:D(h)}
	\Hilb(D(h),\sqrt{q})
	\le \Hilb(A(h)/(x_1,x_n),\sqrt{q})\left\{n(n-1)-(2n-1)\right\}q^2
	\le (n^2-3n+1)q^2[n-2]_q!.
\end{equation}

\begin{proof}[Proof of Proposition $\ref{prop:not-generate}$]
It follows from Lemma \ref{lemm:decomposition_H*}, \eqref{eq:A(h)}, \eqref{eq:B(h)}, \eqref{eq:C(h)}, and \eqref{eq:D(h)} that 
\begin{align*}
	\Hilb(\mathcal{R}(h),\sqrt{q})
	&\le (1+q)^2[n-2]_q!+2(n-1)(q+q^2)[n-2]_q!+(n^2-3n+1)q^2[n-2]_q!\\
	&=(1+2nq+n(n-1)q^2)[n-2]_q!.
\end{align*}
The coefficient of $q^{n-3}$ in the last term above is less than that of $P_n(q)$ in Lemma \ref{lem:Pn=Qn} by $n(n-3)/2$, proving the proposition. 
\end{proof}

\section{Sufficiency} \label{sect:sufficiency}
The purpose of this section is devoted to the proof of the sufficiency of Theorem \ref{thm:main1}. 
Indeed, using an idea in \cite{lin22}, we will show that when the Hessenberg function $h$ is of the form \eqref{eq:h},
$X(h)$ is a fiber bundle over a compact smooth toric variety with a product of flag varieties as the fiber.
This implies that the cohomology ring of $X(h)$ is generated in degree two because so are the cohomology rings of the base space and the fiber.

Let $a,b\in [n]$ with $a < b$. W denote by $\Fl_{[a,b]}(n)$ the partial flag variety consisting of a sequence of linear subspaces in $\C^n$ with consecutive dimensions $a,a+1,\dots,b$.  
We also denote $\Fl(\C^n)$ simply by $\Fl(n)$.  
We consider a map 
\begin{equation} \label{eq:fibration}
\pi_{[a,b]}\colon \Fl(n)\to \Fl_{[a,b]}(n)
\end{equation}
defined by 
\[
	\pi_{[a,b]}(V_1\subset V_2\subset \cdots \subset V_n)
	:=(V_a\subset V_{a+1}\subset \cdots \subset V_b). 
\]
The inverse image of the partial flag $(V_a\subset V_{a+1}\subset \cdots \subset V_b)$ by $\pi_{[a,b]}$ is identified with the set of pairs of partial flags in $\C^n$:
\[
	F_{[a,b]}:=\{ \left((V_1\subset \cdots \subset V_a), (V_b\subset \cdots \subset V_n)\right)\}.
\] 
Since $V_a$ and $V_b$ are fixed,  the former elements above form a complete flag variety in $V_a$, that is isomorphic to $\Fl(a)$.
Taking quotients by $V_b$ for the latter flags above, one sees that they form a variety isomorphic to $\Fl(n-b)$.
Therefore, the map $\pi_{[a,b]}$ in \eqref{eq:fibration} provides a fibration with the fiber $F_{[a,b]}$ isomorphic to $\Fl(a)\times \Fl(n-b)$:
\[
	F_{[a,b]}\to \Fl(n) \xrightarrow{\pi_{[a,b]}} \Fl_{[a,b]}(n).
\]

For our $h$, $V_\bullet=(V_1\subset V_2\subset \cdots \subset V_n)\in \Fl(n)$ is in $X(h)$ if and only if  
\[
	SV_k \subset
	\begin{cases}
		V_{a+1}\quad &(k\le a)\\
		V_{k+1}\quad &(a\le k\le b-1)\\
		V_n=\C^n\quad&(b\le k\le n)
	\end{cases}
\]
Therefore, if $V_\bullet$ is in $X(h)$, then $\pi_{[a,b]}(V_\bullet)=(V_a\subset V_{a+1}\subset \cdots\subset V_b)$ satisfies the condition 
\begin{equation} \label{eq:partial_flag_condition}
	SV_k\subset V_{k+1}\quad (a\le \forall k\le b-1).  
\end{equation}
Conversely, if a partial flag $(V_a\subset V_{a+1}\subset \cdots\subset V_b)$ satisfies the condition \eqref{eq:partial_flag_condition},
then any complete flag $V_\bullet\in \Fl(n)$ extending this partial flag is in $X(h)$.
Indeed, $SV_k\subset V_{a+1}$ for $k\le a$ is satisfied for any choice of $V_k$ because $V_k\subset V_a$ for $k\le a$ and $SV_a\subset V_{a+1}$.
Moreover, $SV_k\subset V_n=\C^n$ for $b\le k\le n$ is trivially satisfied for any choice of $V_k$.  
Therefore, if we set
\[
	Y_{[a,b]}:=\{(V_a\subset \cdots\subset V_b)\in \Fl_{[a,b]}(n)\mid SV_k\subset V_{k+1}\quad (a\le\forall k\le b-1)\}
\]
then $\pi_{[a,b]}(X(h))=Y_{[a,b]}$ and $\pi_{[a,b]}$ restricted to $X(h)$, also denoted by $\pi_{[a,b]}$,
provides a fibration with fiber $F_{[a,b]}$:
\[
	F_{[a,b]}\to X(h)\xrightarrow{\pi_{[a,b]}} Y_{[a,b]}.
\]

\begin{lemma}
$Y_{[a,b]}$ is a compact smooth toric variety of dimension $n-1$.  
\end{lemma}

\begin{proof}
%
%
Since $\pi_{[a,b]}\colon X(h)\to Y_{[a,b]}$ is a fibration and $X(h)$ is a compact smooth variety, $Y_{[a,b]}$ is also a compact smooth variety.  Moreover, since the fiber $F_{[a,b]}$ is isomorphic to $\Fl(a)\times \Fl(n-b)$, it follows from Theorem~\ref{thm:MPS}(2) that 
\begin{align*}
	\dim Y_{[a,b]}&=\dim X(h)-\dim F_{[a,b]}\\
	&=\Big(\frac{1}{2}a(a+1)+b-a-1+\frac{1}{2}(n-b)(n-b+1)\Big)-\Big(\frac{1}{2}a(a-1)+\frac{1}{2}(n-b)(n-b-1)\Big)\\
	&=n-1.
\end{align*}

We shall prove that the action of $(\C^*)^n$ on $Y_{[a,b]}$ has an orbit of dimension $n-1$, which implies that $Y_{[a,b]}$ is a toric variety because it is a smooth variety of dimension $n-1$.  
We take our semisimple matrix $S$ to be a diagonal matrix, so that the diagonal entries denoted by $s_1,\dots,s_n$ are mutually distinct.  Let $\g={}^t(g_1,\dots,g_n)\in (\C^*)^n$. Then vectors $\g,S\g,\dots,S^{k-1}\g$ for $1\le k\le n$ are linearly independent because $s_1,\dots,s_n$ are mutually distinct.  Therefore, the linear subspace $V_k(\g)$ spanned by those vectors are of dimension $k$ and $SV_k(\g)\subset V_{k+1}(\g)$ for $1\le k\le n-1$. Hence, we have 
\[
	V(\g):=\Big(V_a(\g)\subset V_{a+1}(\g)\subset \cdots\subset V_b(\g)\Big)\in Y_{[a,b]}\quad \text{for any $\g\in (\C^*)^n$.}
\]  

Let $\1:={}^t(1,\dots,1)$.  Then $V_k(\g)=\g V_k(\1)$ for any $k$ and hence $V(\g)=\g V(\1)$.
This shows that the set $\{V(\g)\mid \g\in (\C^*)^n\}$ is the $(\C^*)^n$-orbit of $V(\1)$.
If $g_p=g_q$ for any $p$ and $q$, then it is obvious that $\g V(\1)=V(\1)$.
The converse is also true.  In fact, it is true that if $\g V_k(\1)=V_k(\1)$ for some $k\le n-1$,
then $g_p=g_q$ for any $p$ and $q$.  
The proof is as follows.
For $k\le n-1$, we set
\[
	\Lambda_k:=\{I\mid I\subset [n],\ |I|=k\}.
\]
For $I\in \Lambda_k$, we denote by $d_I(\g)$ the determinant of the submatrix formed by all $i$-th rows in $[\g,S\g,\dots,S^{k-1}\g]$ for $i\in I$.
Then 
\[
	[d_I(\g)]_{I\in \Lambda_k}\in \C P^{\binom{n}{k}-1}
\]
is the Pl\"ucker coordinate of $V_k(\g)$.
Since $d_I(\g)=\g_Id_I(\1)$,
where $\g_I:=\prod_{i\in I}g_i$, we have $[d_I(\g)]_{I\in \Lambda_k}=[\g_Id_I(\1)]_{I\in \Lambda_k}$.
Here, $d_I(\1)\not=0$ for any $I\in \Lambda_k$ because $d_I(\1)$ is the Vandermonde's determinant of $s_i$'s $(i\in I)$ and all $s_i$'s are mutually distinct.
It follows that $[d_I(\g)]_{I\in \Lambda_k}=[d_I(\1)]_{I\in \Lambda_k}$
if and only if $\g_I=\g_J$ for any $I,J\in \Lambda_k$,
and this means that $g_p=g_q$ for any $p$ and $q$ since $k\le n-1$.
Therefore, the $(\C^*)^n$-orbit of $V(\1)$ is of dimension $n-1$.  
\end{proof}

\medskip
\noindent
\textbf{Acknowledgment.}

We thank Yunhyung Cho for his help on moment map, Jan-Li Lin for informing us of his paper \cite{lin22}, and Tatsuya Horiguchi for his useful comment on Section~\ref{sect:sufficiency}. Masuda was supported in part by JSPS Grant-in-Aid for Scientific Research 22K03292 and a HSE University Basic Research Program. This work was partly supported by Osaka Central Advanced Mathematical Institute (MEXT Joint Usage/Research Center on Mathematics and Theoretical Physics JPMXP0619217849).

\end{document}